\documentclass[11pt,a4paper]{article}
\usepackage{tikz,amstext,amsbsy,amsfonts,amsmath,amsthm}
\usetikzlibrary{intersections,through,calc}
\usepackage{amsfonts,amsmath}

\newcommand{\fhi}{\varphi}

\newcommand{\norm}[1]{\|#1\|}
\newcommand{\enorm}[1]{\norm{#1}_2}
\newcommand{\ipr}[2]{\left\langle #1, #2 \right\rangle}

\newcommand{\numbersystem}[1]{\mathbb{#1}}

\newcommand{\R}{\numbersystem{R}}

\newcommand{\abs}[1]{\lvert#1\rvert}

\newcommand{\card}[1]{\lvert#1\rvert}

\newtheorem{theorem}{Theorem}

\newtheorem{corollary}[theorem]{Corollary}

\newtheorem{lemma}[theorem]{Lemma}

\newtheorem{proposition}[theorem]{Proposition}

\def\lab(#1)#2{\put(#1){\makebox(0,0)[c]{#2}}}

\begin{document}
\title{The Gilbert Arborescence Problem}
\author{M.~G.~Volz\thanks{TSG Consulting, 350 Collins Street Melbourne,
Victoria 3000, Australia.}\and M.~Brazil\thanks{Department of Electrical and Electronic Engineering, The University of Melbourne, Victoria 3010,
Australia.} \and C.~J.~Ras\footnotemark[2] \and K.~J.~Swanepoel\thanks{Department of Mathematics, London School of Economics and Political
Science, WC2A 2AE London, England.} \and  D.~A.~Thomas\thanks{Department of Mechanical Engineering, University of Melbourne, Victoria 3010,
Australia.}}

\date{}
\maketitle
\begin{abstract}
We investigate the problem of designing a minimum cost flow network interconnecting $n$ sources and a single sink, each with known locations in
a normed space and with associated flow demands. The network may contain any finite number of additional unprescribed nodes from the space;
these are known as the Steiner points. For concave increasing cost functions, a minimum cost network of this sort has a tree topology, and hence
can be called a Minimum Gilbert Arborescence (MGA). We characterise the local topological structure of Steiner points in MGAs, showing, in
particular, that for a wide range of metrics, and for some typical real-world cost-functions, the degree of each Steiner point is $3$.\\

\end{abstract}

\section{Introduction}
The \emph{Steiner Minimum Tree (SMT) problem} asks for a shortest network spanning a given set of nodes (\emph{terminals}) in a given metric
space. It differs from the minimum spanning tree problem in that additional nodes, referred to as \emph{Steiner points}, can be included to
create a spanning network that is shorter than would otherwise be possible. In this paper we consider the geometric version of this problem,
where the metric space is  a normed vector space, and the Steiner points can be any points in that space (as opposed to the network version of
the SMT problem where the Steiner points are restricted to being vertices of a given network). This geometric version of the SMT problem is a
fundamental problem in physical network design optimisation, and has numerous applications, including the design of telecommunications or
transport networks for the problem in the Euclidean plane (the $l_2$ metric), and the physical design of microchips for the problem in the
rectilinear plane (the $l_1$ metric)~\cite{hwang-1992}.

Gilbert~\cite{gilbert-1967} proposed a generalisation of the SMT problem whereby symmetric non-negative \emph{flows} are assigned between each
pair of terminals. The aim is to find a least cost network interconnecting the terminals, where each edge has an associated total flow such that
the flow conditions between terminals are satisfied, and Steiner points satisfy Kirchhoff's rule (ie, the net incoming and outgoing flows demanded from
each Steiner point are equal). The cost of an edge is its length multiplied by a non-negative \emph{weight}. The weight is determined by a given
function of the total flow being routed through that edge, where the weight function satisfies conditions such as being non-negative,
non-decreasing, triangular and concave. These conditions will be made explicit in Section~2.2. The \emph{Gilbert network problem} (GNP) asks for
a minimum-cost network spanning a given set of terminals with given flow demands and a given weight function.

A variation on this problem that we will show to be a special case of the GNP occurs when the  terminals consist of $n$ sources and a unique
sink, and all flows not between a source and the sink are zero. This problem is of intrinsic interest as a natural restriction of the GNP; it is
also of interest for its many applications to areas such as drainage networks~\cite{lee-1976}, gas pipelines~\cite{bhaskaran-1979}, and
underground mining networks~\cite{brazil-2000}.

If the weight function is concave and increasing, the resulting minimum network has a tree topology, and provides a directed path from each
source to the sink. Such a network can be called an \emph{arborescence}, and we refer to this special case of the GNP as the \emph{Gilbert
arborescence problem} (GAP). Traditionally, the term `arborescence' has been used to describe a rooted tree providing directed paths from the
unique root (source) to a given set of sinks. Here we are interested in the case where the flow directions are reversed, i.e. flow is from $n$
sources to a unique sink. It is clear, however, that the resulting weights for the two problems are equivalent, hence we will continue to use
the term `arborescence' for the latter case. Moreover, if we take the sum of these two cases, and rescale the flows (dividing flows in each
direction by $2$), then again the weights for the total flow on each edge are the same as in the previous two cases, and the flows are
symmetric. This justifies our claim that the GAP can be treated as a special case of the GNP. It will be convenient, however, for the remainder
of this paper to think of arborescences as networks with a unique sink.

A \emph{minimum Gilbert arborescence} (MGA) is a (global) minimum-cost arborescence for a given set of terminals and flow demands, and a given cost
function. All flows in the network are directed towards the unique sink. In this paper we investigate the local topological structure of Steiner
points in MGAs, over smooth norms and some typical cost-functions. The analysis of the local structure of vertices of Steiner trees in spaces with non-smooth norms is much more difficult. Even in the classical Steiner tree problem in non-smooth norms, even though there is a general (abstract) characterisation of the local structure \cite{swanepoel-2007}, it is not easy to use this characterisation in specific instances, and has only been done in a few special cases. Although the abstract characterisation for non-smooth norms has been generalised to MGAs in the thesis of Marcus Volz \cite{thesis}, we cannot at present give any concrete application to a specific norm. However, considering only the smooth case in this paper is not a significant restriction, since any non-smooth norm can be approximated to within any required degree of accuracy by a suitable smoothing.

In the optimal design of underground mining tunnel-systems the weight function is usually linear, and the norm is non-Euclidean since there is a constraint on the gradient of the edges \cite{alford}. Although there have recently been many significant developments in the optimal design of gradient-constrained mining networks, a generalisation which includes flow (in this case the flow is the rate of the mass of ore transported along the link) is in need of further mathematical advancement. This underdevelopment of geometric flow-dependent Steiner network algorithms is not due to a lack of important applications, and therefore probably has more to do with the difficulty of the problem.

Finding constraints on the topological structure of MGAs is an essential step towards the goal of producing exact algorithms. There exists a simple generic (and intuitive) algorithmic-framework for the exact construction of many versions of geometric Steiner networks, including MGAs. This framework, which is exemplified by the highly efficient GeoSteiner package for the classical Steiner tree problem, proceeds by constructing every \textit{feasible} topology spanning the terminals and Steiner points, and then finding the optimal locations of the Steiner points with respect to each topology. As evidenced by GeoSteiner, strong local constraints on the set of feasible topologies can significantly reduce the average running time (GeoSteiner runs efficiently on instances of thousands of terminals).

In Section~2 we specify the nature of the weight function that we consider in this paper, and formally define minimum Gilbert networks and
Gilbert arborescences in Minkowski spaces (which generalise Euclidean spaces). In Section~3 we give a general topological characterisation of
Steiner points in such networks, for smooth Minkowski spaces. We then apply this characterisation, in Section~4, to the smooth Minkowski plane
with a linear weight function to show that in this case all Steiner points have degree $3$. In Section~5 we derive a similar result in higher
dimensional Euclidean spaces for a slightly more general class of weight functions.

\section{Preliminaries}

\subsection{Minkowski spaces and Steiner trees}

The cost functions for the networks we consider in this paper make use of more general norms than simply the Euclidean norm. Hence, we introduce
a generalisation of Euclidean spaces, namely finite-dimensional normed spaces or Minkowski spaces. See \cite{Thompson} for an introduction to
Minkowski geometry.

A \emph{Minkowski space} (or \emph{finite-dimensional Banach space}) is $\R^n$ endowed with a \emph{norm} $\norm{\cdot}$, which is a function
$\norm{\cdot}:\R^n\to\R$ that satisfies
\begin{itemize}
\item $\norm{x}\geq 0$ for all $x\in \R^n$, $\norm{x}=0$ only if $x=0$, \item $\norm{\alpha x}=\abs{\alpha}\norm{x}$ for all $\alpha\in\R$ and
$x\in \R^n$, and \item $\norm{x+y}\leq\norm{x}+\norm{y}$.
\end{itemize}

We use $\norm{\cdot}_2$ to denote the Euclidean ($l_2$) norm.

We now discuss some aspects of the SMT problem, since this is a special case of the GNP, where all flow demands are zero (which is equivalent to the
weights on the edges being positive constants). Our terminology for the SMT problem is based on that used in~\cite{hwang-1992}. Let $T$ be a network
interconnecting a set $N = \{p_1,\ldots,p_n\}$ of points, called \emph{terminals}, in a Minkowski space. Vertices of $T$ which are not terminals
are called \emph{Steiner points}, and can consist of any points from the space. Let $G(T)$ denote the \emph{topology} of $T$, i.e. $G(T)$
represents the graph structure of $T$ but not the embedding of the Steiner points. Then $G(T)$ for a shortest network $T$ is necessarily a tree,
since if a cycle exists, the length of $T$ can be reduced by deleting an edge in the cycle. A network with a tree topology is called a
\emph{tree}, its links are called \emph{edges}, and its nodes are called \emph{vertices}. An edge connecting two vertices $a,b$ in $T$ is
denoted by $ab$, and its  length by $\norm{a-b}$.

%The \emph{shrinking} of an edge in $T$ is the operation of deleting an edge and collapsing its two endpoints to a single point.
The \emph{splitting} of a vertex is the operation of disconnecting two edges $av,bv$ from a vertex $v$ and connecting $a,b,v$ to a newly created
Steiner point. %A \emph{degeneracy} of a topology $G(T)$ is another topology that can be obtained by shrinking edges of $G(T)$.
Furthermore, though the positions of terminals are fixed, Steiner points can be subjected to arbitrarily small movements provided the resulting
network is still connected. Such movements are called \emph{perturbations}, and are useful for examining whether the length of a network is
minimal.

A \emph{Steiner tree} (ST) is a tree whose length cannot be shortened by a small perturbation of its Steiner points, even when splitting is
allowed. By convexity, an ST is a minimum-length tree for its given topology. A \emph{Steiner minimum tree} (SMT) is a shortest tree among all
STs, over all topologies and all possible positions of Steiner points in the space. For many Minkowski spaces bounds are known for the maximum
possible degree of a Steiner point in an ST, giving useful restrictions on the possible topology of an SMT. For example, in Euclidean space of
any dimension every Steiner point in an ST has degree three. Given a set $N$ of terminals, the \emph{Steiner problem} (or \emph{Steiner Minimum
Tree problem}) asks for an SMT spanning $N$.

\subsection{Gilbert flows}

Gilbert~\cite{gilbert-1967} proposed the following generalisation of the Steiner problem in Euclidean space, which we now extend to Minkowski
space. Let $T$ be a network interconnecting a set $N = \{p_1,\ldots,p_n\}$ of $n$ terminals in a Minkowski space. For each pair $p_i,p_j,\;i\neq
j$ of terminals, a non-negative flow demand $t_{ij} = t_{ji}$ is given. The cost of an edge $e$ in $T$ is $w(t_e)l_e$, where $l_e$ is the length
of $e$, $t_e$ is the total flow being routed through $e$, and $w(\cdot)$ is a unit cost \emph{weight function} defined on $[0,\infty)$
satisfying

\begin{align}
  w(0) &\geq 0\quad\text{and}\quad w(t) > 0  \text{ for all }t>0, \label{eq:cost-ftn-non-neg} \\
  w(t_2) &\geq w(t_1)\quad \text{for all}\quad t_2 > t_1\geq 0, \label{eq:cost-ftn-non-dec} \\
   w(\cdot)&\text{ is a concave function.} \label{eq:cost-ftn-concave}
\end{align}

That the function $w$ is concave means by definition that $-w$ is convex. Conditions ~\eqref{eq:cost-ftn-non-neg} and ~\eqref{eq:cost-ftn-concave} imply the following linearity condition

\begin{align}
  w(t_1+t_2) &\leq w(t_1)+w(t_2)\quad\text{for all}\quad t_1,t_2 > 0. \label{eq:cost-ftn-trnglr}
\end{align}

A network satisfying Conditions~\eqref{eq:cost-ftn-non-neg}, \eqref{eq:cost-ftn-non-dec}, and \eqref{eq:cost-ftn-trnglr} (but not necessarily Condition \eqref{eq:cost-ftn-concave}) is called a \emph{Gilbert network}. For a given edge $e$ in $T$, $w(t_e)$ is
called the \emph{weight} of $e$, and is also denoted simply by $w_e$. The \emph{total cost} of a Gilbert network $T$ is the sum of all edge
costs, i.e.
\begin{eqnarray*}
  C(T) &=& \sum_{e\in E}w(t_e)l_e
\end{eqnarray*}
\noindent where $E$ is the set of all edges in $T$. A Gilbert network $T$ is a \emph{minimum Gilbert network} (MGN), if $T$ has the minimum cost
of all Gilbert networks spanning the same point set $N$, with the same flow demands $t_{ij}$ and the same cost function $w(\cdot)$. By the
arguments of~\cite{cox-1998}, an MGN always exists in a Minkowski space when Conditions~\eqref{eq:cost-ftn-non-neg}, \eqref{eq:cost-ftn-non-dec}, and \eqref{eq:cost-ftn-trnglr}
are assumed for the weight function.

Conditions~\eqref{eq:cost-ftn-non-neg}, \eqref{eq:cost-ftn-non-dec}, and \eqref{eq:cost-ftn-trnglr} ensure that the weight function is non-negative, non-decreasing and
triangular, respectively. These are natural conditions for most applications. Unfortunately, these conditions alone do not guarantee
that a minimum Gilbert network is a tree. To show this, we now give an example of a Gilbert network problem with two sources and one sink in the
Euclidean plane, where there exists a \textit{split-route flow} (i.e. some vertex has at least two out-going edges and therefore the network contains a cycle) that has a lower cost than any arborescence.

For this example there are two sources $p_1, p_2$ and a sink $q$ which are the vertices of a triangle $\triangle p_1p_2q$ with edge lengths
$\enorm{p_1-p_2}=1$ and $\enorm{p_1-q}=\enorm{p_2-q}=10$, as illustrated in Figure~\ref{figureEx1}.
\begin{figure}
\begin{center}
\begin{tikzpicture}[scale=1.1, line join=round]

% Define the points
\coordinate (o) at (0,0);
\coordinate [label=left:{$p_1$}] (p1) at (0,0.5);
\coordinate [label=left:{$p_2$}] (p2) at (0,-0.5);
\pgfmathsetmacro{\qxcoord}{sqrt(99.75)}
\coordinate [label=below:{$q$}] (q) at (\qxcoord,0);

% Draw the triangle p_1 p_2 q
\draw [color=black,very thick,-stealth] (p2) -- ($(p2)!0.55!(p1)$) node[right] {$t=1$}; \draw[color=black,very thick] (p2)-- (p1);
\draw [color=black,very thick, -stealth] (p1) -- ($(p1)!0.55!(q)$) node[above] {$t=3$}; \draw[color=black,very thick] (p1)-- (q);
\draw [color=black,very thick, -stealth] (p2) -- ($(p2)!0.55!(q)$) node[below] {$t=3$}; \draw[color=black,very thick] (p2)-- (q);

% Draw the points
\foreach \point in {p1,p2,q} 
\path [draw=black, fill=white] (\point) circle (2pt); 

\end{tikzpicture}
\caption{An example where split-routing is cheaper\label{figureEx1}}
\end{center}
\end{figure}
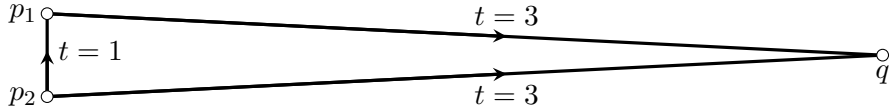
The flows demanded from $p_1$ and $p_2$ are $2$ and $4$, respectively. The weight function is $w(t)=\lceil (3t+1)/2\rceil$, i.e., $(3t+1)/2$ rounded up to
the nearest integer. This function is positive, non-decreasing and triangular, but not concave. For the example we only need the following
values:
\[
\begin{array}{c|ccccc} t & 1 & 2 & 3 & 4 & 6 \\ \hline w(t) & 2 & 4 & 5 & 7 & 10
\end{array}\]
Routing $1$ unit of the flow from $p_2$ via $p_1$ to $q$ gives a Gilbert network (Figure~\ref{figureEx1}) of total cost
$$w(1)\enorm{p_1-p_2} + w(3)\enorm{p_1-q} + w(3)\enorm{p_2-q} = 102.$$

For a Gilbert arborescence we route the flows from $p_1$ and $p_2$ to $q$ via some point $s$ as in Figure~\ref{figureEx2}.
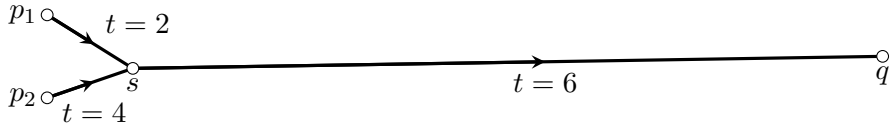
\begin{figure}
\begin{center}
\begin{tikzpicture}[scale=1.1, line join=round]

% Define the points
\coordinate (o) at (0,0);
\coordinate [label=left:{$p_1$}] (p1) at (0,0.5);
\coordinate [label=left:{$p_2$}] (p2) at (0,-0.5);
\pgfmathsetmacro{\qxcoord}{sqrt(99.75)}
\coordinate [draw=white,fill=white,label=below:{$q$}] (q) at (\qxcoord,0);
\pgfmathsetmacro{\pxcoord}{-0.035*sqrt(39)};
\coordinate (p) at (\pxcoord,-0.165);
\pgfmathsetmacro{\mxcoord}{2.5/sqrt(39)};
\coordinate (m) at (\mxcoord,0);

% Draw the weighted Simpson line
\draw [name path=simpson,color=white] (p) -- (q); 

% Draw the circle
\node [name path=Circle, draw,circle through=(p1),color=white] at (m) {}; 

\path [name intersections={of=simpson and Circle}]; 
\coordinate [label=below:$s$] (s) at (intersection-2); 

% Draw the Gilbert network
\begin{scope}[color=black,very thick]
\draw [-stealth] (p1) -- ($(p1)!0.55!(s)$) node[above right] {$t=2$}; \draw (p1)-- (s);
\draw [-stealth] (p2) -- ($(p2)!0.55!(s)$) node[below=4pt] {$t=4$}; \draw (p2)--(s);
\draw [-stealth] (s) -- ($(s)!0.55!(q)$) node[below] {$t=6$}; \draw (s) -- (q);
\end{scope}

% Draw the points
\foreach \point in {p1,p2,s,q} 
\path [draw=black, fill=white] (\point) circle (2pt); 

\end{tikzpicture}
\end{center}
\caption{The minimum Gilbert arborescence\label{figureEx2}}
\end{figure}
We calculate a minimum Gilbert arborescence by using the weighted Melzak algorithm as described in \cite{gilbert-1967}. To construct the
weighted Fermat-Torricelli point $s$, first construct the unique point $p$ outside $\triangle p_1p_2q$ such that $\enorm{p-p_1}=0.7$ and
$\enorm{p-p_2}=0.4$ (Figure~\ref{figureEx3}).
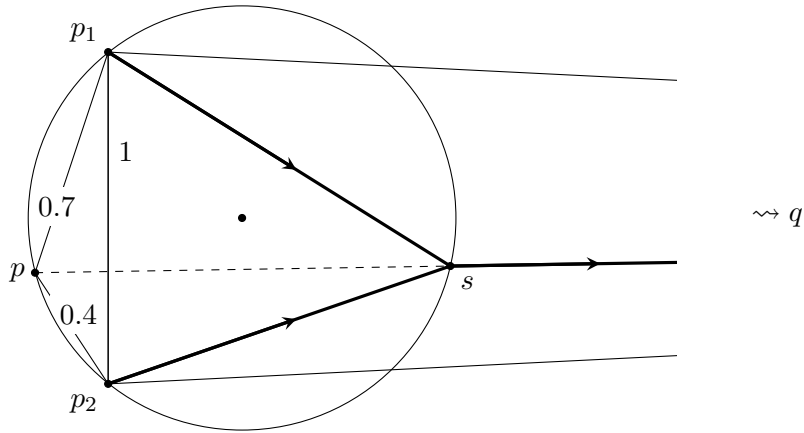
\begin{figure}
\begin{center}
\begin{tikzpicture}[scale=4.4, line join=round]

% Define the points
\coordinate (o) at (0,0);
\coordinate [label=above left:{$p_1$}] (p1) at (0,0.5);
\coordinate [label=below left:{$p_2$}] (p2) at (0,-0.5);
\pgfmathsetmacro{\qxcoord}{sqrt(99.75)}
\coordinate [label=below:{$q$}] (q) at (\qxcoord,0);
\coordinate (r) at (-0.8,0);
\draw [white] (r) circle (1pt);
\pgfmathsetmacro{\pxcoord}{-0.035*sqrt(39)};
\coordinate [label=left:{$p$}] (p) at (\pxcoord,-0.165);
\pgfmathsetmacro{\mxcoord}{2.5/sqrt(39)};
\coordinate (m) at (\mxcoord,0);

\draw (2,0) node{$\rightsquigarrow q$};

% Clip
\clip (-0.5,-0.8) rectangle (1.7,0.8);

% Draw the triangle p_1 p_2 p
\draw (p1)--(p2)--(p)node [below right,pos=0.8,fill=white] {$0.4$}-- (p1) node [pos=0.3,fill=white] {$0.7$};

% Draw the triangle p_1 p_2 q
\draw (p1)--(p2) node [right,pos=0.3] {$1$} -- (q)  -- (p1);

% Draw the weighted Simpson line
\draw [name path=simpson, dashed] (p) -- (q); 

% Draw the circle
\node [name path=Circle, draw,circle through=(p1)] at (m) {}; 

\path [name intersections={of=simpson and Circle}]; 
\coordinate [label=below right:$s$] (s) at (intersection-2); 

% Draw the Gilbert network
\begin{scope}[color=black,very thick]
\draw [-stealth] (p1) -- ($(p1)!0.55!(s)$); \draw (p1)-- (s);
\draw [-stealth] (p2) -- ($(p2)!0.55!(s)$); \draw (p2)--(s);
\draw [-stealth] (s) -- ($(s)!0.05!(q)$); \draw (s) -- (q);
\end{scope}

% Draw the points
\foreach \point in {p1,p2,q,p,m,s} 
\path [draw=black, fill=black] (\point) circle (0.3pt); 
\end{tikzpicture}
\end{center}
\caption{Constructing the weighted Fermat-Torricelli point\label{figureEx3}}
\end{figure}
Then construct the circumscribed circle of $\triangle pp_1p_2$, which will intersect the so-called \textit{weighted Simpson line} $pq$ in the
required point $s$. Using a little trigonometry, it can be seen that the resulting total cost is
\begin{align*}
&\quad w(2)\enorm{p_1-s}+w(4)\enorm{p_2-s}+w(6)\enorm{s-q}\\
&=w(6)\enorm{p-q}=\sqrt{9982.5+7\sqrt{3890.25}}\\
&=102.074\dots.
\end{align*}
This shows that the split routing we constructed is cheaper than the cheapest non-split routing, so that split routing can be necessary when the weight function is not concave. For the remainder of the paper we assume that the weight function $w$ satisfies Conditions~\eqref{eq:cost-ftn-non-neg}, \eqref{eq:cost-ftn-non-dec}, and  \eqref{eq:cost-ftn-concave}. In this case it is
known \cite{gilbert-1967, thomas-2006} that in the case where there is a single sink there always exists a minimum Gilbert network that is a
Gilbert arborescence. This means that we can (and will) without loss of generality only consider MGAs. (Note that in~\cite{cox-1998},
Condition~(\ref{eq:cost-ftn-trnglr}), which we call the \emph{triangular condition}, was incorrectly interpreted as concavity of the cost
function.)

 The \emph{Gilbert network problem} (GNP) is to find an MGN for a
given terminal set $N$, flow demands $t_{ij}$ and cost function $w(\cdot)$. Since its introduction in~\cite{gilbert-1967}, various aspects of the GNP
have been studied, although the emphasis has been on discovering geometric properties of MGNs
(see~\cite{cox-1998},~\cite{thomas-2006},~\cite{trietsch-1985},~\cite{trietsch-1999}). As in the Steiner problem, additional vertices can be
added to create a Gilbert network whose cost is less than would otherwise be possible, and these additional points are again called
\emph{Steiner points}. A Steiner point $s$ in $T$ is called \emph{locally minimal} if a perturbation of $s$ does not reduce the cost of $T$. A
Gilbert network is called \emph{locally minimal} if no perturbation of the Steiner points reduces the cost of $T$.

The special case of the Gilbert model that is of interest in this work is when $N = \{p_1,\ldots,p_n, q\}$ is a set of terminals in a Minkowski
space, where $p_1,\ldots,p_{n}$ are \emph{sources} with respective positive flow demands $t_1,\ldots,t_{n}$, and $q$ is the \emph{sink}.
All flows are between the sources and the sink; there are no flows between sources. It has been shown in~\cite{thomas-2006} that concavity of
the weight function implies that an MGN of this sort is a tree. Hence we refer to an MGN with this flow structure as a \emph{minimum Gilbert
arborescence} (MGA), and, as mentioned in the introduction, we refer to the problem of constructing such an MGA as the \emph{Gilbert
arborescence problem} (GAP).

If $v_1$ and $v_2$ are two adjacent vertices in a Gilbert arborescence, and the flow is from $v_1$ to $v_2$ then we denote the edge connecting
the two vertices by $v_1v_2$.

\section{Characterisation of Steiner Points}\label{section:characterisation}
In this section, we generalise a theorem of Lawlor and Morgan \cite{lawlor-1994} to give a local characterisation of Steiner points in an MGA.
The characterisation in \cite{lawlor-1994} holds for SMTs, which correspond to the case of MGAs with a constant weight function. Their theorem
is formulated for arbitrary Minkowski spaces with differentiable norm. Our proof is based on the proof of Lawlor and Morgan's theorem given in
\cite{swanepoel-1999}. A generalisation to non-smooth norms is contained in~\cite{swanepoel-2007} for SMTs and in \cite{thesis} for MGAs. Such a
generalisation is much more complicated and involves the use of the subdifferential calculus.

We first introduce some necessary definitions relating to Minkowski geometry, in particular with relation to dual spaces. For more details, see
\cite{Thompson}.

We denote the inner product of two vectors $x,y\in\R^n$ by $\ipr{x}{y}$. For any given norm $\norm{\cdot}$, the dual norm $\norm{\cdot}^\ast$ is
defined as follows:
\[\norm{z}^\ast = \sup_{\norm{x}\leq 1}\ipr{z}{x}.\]

We say that a Minkowski space $(\R^n,\norm{\cdot})$ is {\em smooth} if the norm is differentiable at any $x\neq o$, i.e., if
\[\lim_{t\to 0} \frac{\norm{x+th}-\norm{x}}{t} =: f_x(h)\]
exists for all $x,h\in \R^n$ with $x\neq o$. It follows easily that $f_x$ is a linear operator $f_x:\R^n\to\R$ and so can be represented by a
vector $x^\ast\in\R^n$, called the dual vector of $x$, such that $\ipr{x^\ast}{y}=f_x(y)$ for all $y\in\R^n$, and $\norm{x^\ast}^\ast=1$. In
fact $x^\ast$ is just the gradient of the norm at $x$, i.e., $x^\ast=\nabla\norm{x}$.

More generally, even if the norm is not differentiable at $x$, a vector $x^\ast\in \R^n$ is a {\em dual vector} of $x$ if $x^\ast$ satisfies
$\ipr{x^\ast}{x}=\norm{x}$ and $\norm{x^\ast}^\ast=1$. By the Hahn-Banach separation theorem, each non-zero vector in a Minkowski space has at
least one dual vector. A Minkowski space is then smooth if and only if each non-zero vector has a unique dual vector.

A norm is {\em strictly convex} if $\norm{x}=\norm{y}=1$ and $x\neq y$ imply that $\norm{\frac{1}{2}(x+y)}<1$, or equivalently, that the
\emph{unit sphere} \[ S(\norm{\cdot}) = \{x\in\R^n: \norm{x}=1\}\]  does not contain any straight line segment. A norm $\norm{\cdot}$ is smooth
[strictly convex] if and only if the dual norm $\norm{\cdot}^\ast$ is strictly convex [smooth, respectively].

\begin{theorem}\label{theorem:gilb-arb-charac}
Suppose a smooth Minkowski space $(\R^n,\norm{\cdot})$ is given together with  a weight function $w$ that satisfies
Conditions~\eqref{eq:cost-ftn-non-neg}--\eqref{eq:cost-ftn-concave}, sources $p_1,\dots,p_n\in\R^n$, and a single sink $q\in\R^n$, all different
from the origin $o$. Let the flow demand at $p_i$ be $t_i$. (See Figure~\ref{Fig:th1}.)
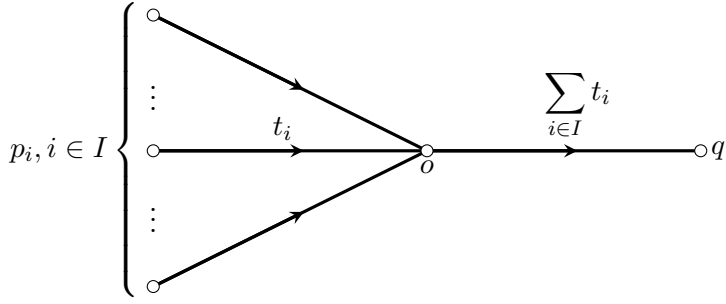
\begin{figure}
\begin{center}
\begin{tikzpicture}[scale=0.9]
%\draw (-4,-2) grid (4,2);

% Define the points
\coordinate [label=below:{$o$}] (x) at (-2,0);
\coordinate [label=right:{$q$}] (o) at (2,0);
\coordinate (p1) at (-6,2);
\coordinate (p2) at (-6,0);
\coordinate (pn) at (-6,-2);

% Draw vertical ellipses
\draw (-6,-0.9) node {$\vdots$};
\draw (-6,0.9) node {$\vdots$};

% Draw the Gilbert network
\begin{scope}[color=black,very thick]
\draw [-stealth] (p1) -- ($(p1)!0.55!(x)$); \draw (p1)-- (x);
\draw [-stealth] (p2) node[left] {$p_i, i\in I\left\{\rule[-1.9cm]{0cm}{2cm}\right.$} -- ($(p2)!0.55!(x)$) node [above left] {$t_i$}; \draw (p2)--(x);
\draw [-stealth] (pn) -- ($(pn)!0.55!(x)$); \draw (pn)-- (x);
\draw [-stealth] (x) -- ($(x)!0.55!(o)$) node [above] {$\displaystyle\sum_{i\in I} t_i$}; \draw (x) -- (o);
\end{scope}

% Draw the points
\foreach \point in {x,o,p1,p2,pn} 
\path [draw=black, fill=white] (\point) circle (2.5pt); 

\end{tikzpicture}
\end{center}
\caption{A Gilbert network with star topology, where $o$ is the origin.} \label{Fig:th1}
\end{figure}
For each $p_i$ let $p_i^\ast$ denote its dual vector, and let $q^\ast$ denote the dual vector of $q$. Then the Gilbert arborescence with edges
$op_i$, $i=1,\dots,n$ and $oq$, where all flows are routed via the Steiner point $o$, is a minimal Gilbert arborescence if and only if
\begin{equation}\label{balancing}
\sum_{i=1}^n w(t_i)p_i^\ast + w(\sum_{i=1}^n t_i)q^\ast = o
\end{equation}
and
\begin{equation}\label{collapsing}
\norm{\sum_{i\in I} w(t_i)p_i^\ast}^\ast\leq w(\sum_{i\in I} t_i)\text{ for all $I\subseteq\{1,\dots,n\}$.}
\end{equation}
\end{theorem}

Note: We think of Condition~\ref{balancing} as a flow-balancing condition at the Steiner point, and Condition~\ref{collapsing} as a condition
that ensures that the Steiner point does not split.

\begin{proof}
$(\Rightarrow)$ We are given that the star is not more expensive than any other Gilbert network with the same sources, sink, flows and weight
function.

In particular, $o$ is the so-called weighted Fermat-Torricelli point of the $n+1$ points $p_1,\dots,p_n,q$ with weights
$t_1,\dots,t_n,\sum_{i=1}^n t_i$, respectively, which implies the balancing condition \eqref{balancing}. We include a self-contained proof for
completeness. If the Steiner point $o$ is moved to $-te$, where $t\in\R$ and $e\in\R^n$ is a unit vector (in the norm), the resulting
arborescence is not better, by the assumption of minimality. Therefore, the function
\begin{align*}
 \fhi_e(t) &=\sum_{i=1}^n w(t_i)(\norm{p_i+te}-\norm{p_i})\\ &\quad+w(\sum_{i=1}^n t_i)(\norm{q+te}-\norm{q})\geq 0
 \end{align*}
attains its minimum at $t=0$. For $t$ in a sufficiently small neighbourhood of $0$, $p_i+te\neq o$ and $q+te\neq o$, hence $\fhi_e$ is
differentiable. Therefore,
\begin{align*}
0&=\fhi_e'(0) =\lim_{t\to 0} \left(\sum_{i=1}^n w(t_i)\frac{\norm{p_i+te}-\norm{p_i}}{t}\right.\\
& \qquad\qquad \qquad +\left.w(\sum_{i=1}^n t_i)\frac{\norm{q+te}-\norm{q}}{t}\right)\\
&= \sum_{i=1}^n w(t_i)\ipr{p_i^\ast}{e}+w(\sum_{i=1}^n t_i)\ipr{q^\ast}{e}\\
&= \ipr{\sum_{i=1}^n w(t_i)p_i^\ast+w(\sum_{i=1}^n t_i)q^\ast}{e}.
\end{align*}
Since this holds for all unit vectors $e$, \eqref{balancing} follows.

To show \eqref{collapsing} for each $I\subseteq\{1,\dots,n\}$, we may assume without loss of generality that $I\neq\emptyset$ and
$I\neq\{1,\dots,n\}$. Consider the Gilbert network obtained by splitting the Steiner point into two points $o$ and $+te$ ($t\in\R$, $e$ a unit
vector) as follows. Each $p_i$, $i \notin I$, is still adjacent to $o$ with flow demand $t_i$, and $q$ is joined to $o$ with flow demand $\sum_{i=1}^n t_i$,
but now each $p_i$, $i\in I$, is adjacent to $te$ with flow demand $t_i$, and $te$ is adjacent to $o$ with flow demand $\sum_{i\in I} t_i$, as shown in
Figure~\ref{Fig:split}.
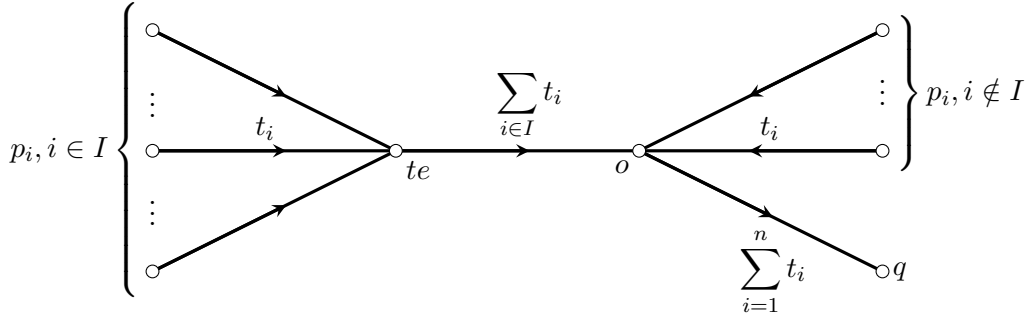
\begin{figure}
\begin{center}
\begin{tikzpicture}[scale=0.8]
%\draw (-4,-2) grid (4,2);

% Define the points
\coordinate [label=below right:{$te$}] (x) at (-2,0);
\coordinate [label=below left:{$o$}] (o) at (2,0);
\coordinate (p1) at (-6,2);
\coordinate [label=left:{$p_i, i\in I\left\{\rule[-1.9cm]{0cm}{2cm}\right.$}] (p2) at (-6,0);
\coordinate (p3) at (6,2);
\coordinate (p5) at (6,0);
\draw ($(p5)+(0,1)$) node [right]{$\left.\rule[-0.4cm]{0cm}{1.6cm}\right\}p_i, i\notin I$};
\coordinate (pn) at (-6,-2);
\coordinate [label=right:{$q$}] (q) at (6,-2);

% Draw the vertical ellipses
\draw (-6,-0.9) node {$\vdots$} ; 
\draw (-6,0.9) node {$\vdots$} ; 
\draw (6,1.1) node {$\vdots$} ;

% Draw the Gilbert network
\begin{scope}[color=black,very thick]
\draw [-stealth] (p1) -- ($(p1)!0.55!(x)$); \draw (p1)-- (x);
\draw [-stealth] (p2) -- ($(p2)!0.55!(x)$) node [above left] {$t_i$}; \draw (p2)--(x);
\draw [-stealth] (p3) -- ($(p3)!0.55!(o)$); \draw (p3)--(o);
\draw [-stealth] (p5) -- ($(p5)!0.55!(o)$) node [above right] {$t_i$}; \draw (p5)-- (o);
\draw [-stealth] (pn) -- ($(pn)!0.55!(x)$); \draw (pn)-- (x);
\draw [-stealth] (o) -- ($(o)!0.55!(q)$) node [below] {$\displaystyle\sum_{i=1}^n t_i$}; \draw (o) -- (q);
\draw [-stealth] (x) -- ($(x)!0.55!(o)$) node [above] {$\displaystyle\sum_{i\in I} t_i$}; \draw (x) -- (o);
\end{scope}

% Draw the points
\foreach \point in {x,o,p1,p2,p3,p5,pn,q} 
\path [draw=black, fill=white] (\point) circle (3pt); 

\end{tikzpicture}
\end{center}
\caption{The Gilbert network obtained by splitting the Steiner point $o$.} \label{Fig:split}
\end{figure}
Since the new network cannot be better than the original star, we obtain that for any unit vector $e$, the function
$$
 \psi_e(t)=\sum_{i\in I} w(t_i)(\norm{p_i-te}-\norm{p_i})+w(\sum_{i=1}^n t_i)\abs{t}\geq 0
$$
attained its minimum at $t=0$. Although $\psi_e$ is not differentiable at $0$, we can still calculate as follows:
\begin{align*}
0 &\leq \lim_{t\to 0+}\frac{\psi_e(t)}{t}\\
&= \lim_{t\to0+}\sum_{i\in I} w(t_i)\frac{\norm{p_i-te}-\norm{p_i}}{t}+w(\sum_{i=1}^n t_i)\\
&= \ipr{\sum_{i\in I} w(t_i)p_i^\ast}{-e}+w(\sum_{i=1}^n t_i).
\end{align*}
Therefore, $\ipr{\sum_{i\in I} w(t_i)p_i^\ast}{e}\leq w(\sum_{i=1}^n t_i)$ for all unit vectors $e$, and \eqref{collapsing} follows from the
definition of the dual norm.

\noindent $(\Leftarrow)$ Now assume that $p_1^\ast,p_n^\ast,q$ are dual unit vectors that satisfy \eqref{balancing} and \eqref{collapsing}.
Consider an arbitrary Gilbert arborescence $T$ for the given data. For each $i$, let $P_i$ be the path in $T$ from $p_i$ to $q$, i.e., $P_i =
x_{1}^{(i)}x_2^{(i)}\dots x_{k_i}^{(i)}$, where $x_1^{(i)}=p_i, x_{k_i}^{(i)}=q$, and $x_j^{(i)} x_{j+1}^{(i)}$ are distinct edges of $T$ for
$j=1,\dots,k_i-1$. For each edge $e$ of $T$, let $S_e=\{i:\text{$e$ is on path $P_i$}\}$. Then the flow on $e$ is $\sum_{i\in S_e} t_i$ and the
total cost of $T$ is
\[\sum_{\substack{e=xy\text{ is}\\ \text{an edge of $T$}}} w(\sum_{i\in S_e} t_i)\norm{x-y}.\]
The cost of the star is
\begin{align*}
&\quad\sum_{i=1}^n w(t_i)\norm{p_i}+w(\sum_{i=1}^n t_i)\norm{q} \\
&= \sum_{i=1}^n w(t_i)\ipr{p_i^\ast}{p_i}+w(\sum_{i=1}^n t_i)\ipr{q^\ast}{q} \\
&= \sum_{i=1}^n w(t_i)\ipr{p_i^\ast}{p_i-q}\qquad\text{by \eqref{balancing}}\\
&= \sum_{i=1}^n w(t_i) \sum_{j=1}^{k_i-1}\ipr{p_i^\ast}{x_j^{(i)}-x_{j+1}^{(i)}}
\end{align*}
\begin{align*}
&= \sum_{\substack{e=xy\text{ is}\\ \text{an edge of $T$}}} \ipr{\sum_{i\in S_e} w(t_i) p_i^\ast}{x-y}\\
&\leq \sum_{\substack{e=xy\text{ is}\\ \text{an edge of $T$}}}\norm{\sum_{i\in S_e} w(t_i) p_i^\ast}^\ast\norm{x-y}\\
&\leq \sum_{\substack{e=xy\text{ is}\\ \text{an edge of $T$}}} w(\sum_{i\in S_e} t_i)\norm{x-y} \quad\text{by \eqref{collapsing}}.
\end{align*}
This concludes the proof. \end{proof}

Note that the necessity of the conditions \eqref{eq:cost-ftn-non-neg}, \eqref{eq:cost-ftn-non-dec}, \eqref{eq:cost-ftn-trnglr},
\eqref{balancing} and \eqref{collapsing} holds even if the weight function is not concave. It is only in the proof of the sufficiency that we
need all minimal Gilbert networks with a single sink to be arborescences.

\section{Degree of Steiner Points in a Minkowski plane with linear weight function}%\label{section:degree}
We now apply the characterisation of the previous section in the two\nobreakdash-\hspace{0pt}dimensional case, assuming further that the weight
function is linear: $w(t)=d+ht$, $d>0, h\geq 0$.

\begin{theorem}\label{thm2}
In a smooth Minkowski plane and assuming a linear weight function $w(t)=d+ht$, $d>0, h\geq 0$, a Steiner point in an MGA necessarily has degree
$3$.
\end{theorem}
\begin{proof}
By Theorem~\ref{theorem:gilb-arb-charac}, an MGA with a Steiner point of degree $n+1$ exists in $\R^2$ with a smooth norm $\norm{\cdot}$ if and
only if there exist dual unit vectors $p_1^\ast,\dots,p_n^\ast,q^\ast\in\R^2$ such that
\[ \sum_{i=1}^n (d+ht_i) p_i^\ast + (d+h\sum_{i=1}^n t_i) q^\ast = o \]
and
$$
 \norm{\sum_{i\in I} (d+ht_i)p_i^\ast}^\ast \leq d+h\sum_{i\in I} t_i \quad\text{for all } I\subseteq\{1,\dots,n\}.
$$
Label the $p_i^\ast$ so that they are in order around the dual unit circle. Let $v_i^\ast=(d+ht_i)p_i^\ast$ and $w^\ast=(d+h\sum_{i=1}^n
t_i)q^\ast$. Then the conditions become
\[ v_1^\ast+\dots+v_n^\ast+w^\ast=o,\]
and
\begin{equation}\label{collapsing2}
\norm{\sum_{i\in I} v_i^\ast}^\ast\leq d+h\sum_{i\in I} t_i \quad \text{for all } I\subseteq\{1,\dots,n\}.
\end{equation}
Thus we may think of the vectors $v_1^\ast,\dots,v_n^\ast,w^\ast$ as the edges of a convex polygon with vertices $a_j^\ast=\sum_{i=1}^j
v_i^\ast$, $j=0,\dots,n$ in this order (see Figure~\ref{Fig:polygon}).

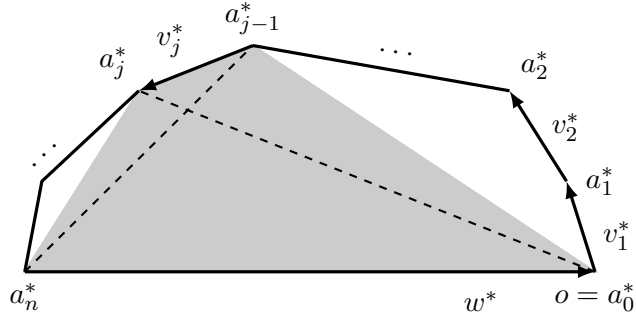
\begin{figure}
\begin{center}
\begin{tikzpicture}[thick, xscale=0.75, yscale=0.6]

%Grid for debugging
%\draw (0,0) grid (10,5);

% shaded quadrilateral
\filldraw[black!20] (0,0) -- (10,0) -- (4,5) -- (2,4) -- cycle;

% polygon
\begin{scope}[very thick]
\draw[-latex] (0,0) node[below] {$a_n^\ast$} -- (10,0) node[below] {$o=a_0^\ast$} node[below=3pt,pos=0.8] {$w^\ast$}; 
\draw[-latex] (10,0) -- (9.5,2) node[right=3pt,pos=0.4] {$v_1^\ast$}; 
\draw[-latex] (9.5,2) node[right=3pt] {$a_1^\ast$} -- (8.5,4) node[right=3pt,pos=0.6] {$v_2^\ast$}; 
\draw (8.5,4) node[above right] {$a_2^\ast$} -- (4,5) node[above, rotate=27, pos=0.4] {$\ddots$}; 
\draw[-latex] (4,5) node[above] {$a_{j-1}^\ast$} -- (2,4) node[above left] {$a_j^\ast$} node[above left=0pt,midway] {$v_j^\ast$}; 
\draw (2,4) -- (0.3,2) node[right=-4pt, rotate=79] {$\ddots$} -- (0,0);
\end{scope}
% Diagonals
\draw[dashed, thick] (0,0) -- (4,5); \draw[dashed, thick] (10,0) -- (2,4);

\end{tikzpicture}
\end{center}
\caption{A polygon with edges corresponding to $v_j^\ast$.}\label{Fig:polygon}
\end{figure}
Assume for the purpose of finding a contradiction that $n > 3$. Then the polygon has at least $4$ sides. Note that the diagonals $a_0^\ast
a_{j}^\ast$ and $a_{j-1}^\ast a_n^\ast$ intersect. Applying the triangle inequality to the two triangles formed by these diagonals and the two
edges $v_j^\ast$ and $w^\ast$ (as illustrated in Figure~\ref{Fig:polygon}), we obtain
\begin{align*}
%&\phantom{\geq}
 \norm{a_j^\ast}^\ast+\norm{a_n^\ast-a_{j-1}^\ast}^\ast
&\geq \norm{v_j^\ast}^\ast+\norm{w^\ast}^\ast\\
&= d+ht_j + d+h\sum_{i=1}^n t_i\\
&= d+h\sum_{i=1}^j t_i + d + h\sum_{i=j}^n t_i\\
&\geq \norm{\sum_{i=1}^j v_i^\ast}^\ast+\norm{\sum_{i=j}^n v_i^\ast}^\ast \qquad\text{by \eqref{collapsing2}}\\
&= \norm{a_j^\ast}^\ast+\norm{a_n^\ast-a_{j-1}^\ast}^\ast.
\end{align*}
Therefore, equality holds throughout, and we obtain equality in the triangle inequality. Since we assume the norm is smooth, the dual norm is
strictly convex and it follows that $v_j^\ast$ and $w^\ast$ are parallel. This holds for all $j=2,\dots,n-1$. It follows that
$p_1^\ast=\dots=p_n^\ast=-q^\ast$. Geometrically this means that the unit vectors $\frac{1}{\norm{p_i}}p_i$ and $-\frac{1}{\norm{q}}q$ all have
the same supporting line on the unit ball. We can think of this condition on the vectors $p_i$ and $q$ as a generalisation of collinearity to
Minkowski space.

Choose a point $s_2$ on the edge $op_i$ such that the line through $s_2$ parallel to $op_n$ intersects the edge $op_1$ in $s_1$, say, with
$s_1\neq o$. See Figure~\ref{Fig:unitball}.
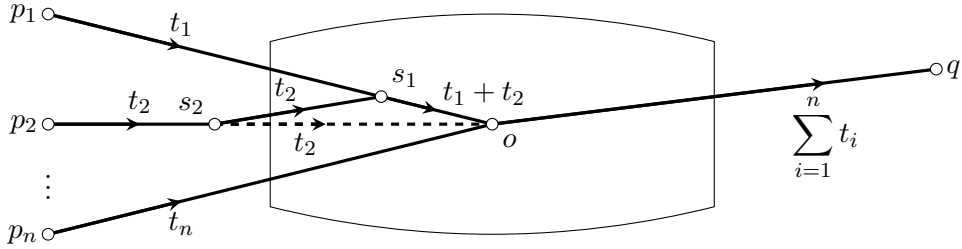
\begin{figure}
\begin{center}
\begin{tikzpicture}[scale=0.73]
% The unit ball
\draw (-4,1.5) parabola bend (0,2) (4,1.5) -- (4,-1.5) parabola bend (0,-2) (-4,-1.5) -- cycle;

% Define the points
\coordinate [label=below right:{$o$}] (o) at (0,0);
\coordinate [label=left:{$p_1$}] (p1) at (-8,2);
\coordinate [label=left:{$p_2$}] (p2) at (-8,0);
\coordinate [label=left:{$p_n$}] (pn) at (-8,-2);
\coordinate [label=above right:{$s_1$}] (s1) at (-2,0.5);
\coordinate [label=above left:{$s_2$}] (s2) at (-5,0);
\coordinate [label=right:{$q$}] (q) at (8,1);

% Draw vertical ellipsis
\draw (-8,-1) node {$\vdots$};

% Draw the Gilbert network
\begin{scope}[color=black,very thick]
\draw (p1)-- (o);
\draw [-stealth] (p1) -- ($(p1)!0.4!(s1)$) node [above] {$t_1$};
\draw [-stealth] (s1) -- ($(s1)!0.5!(o)$) node[above right=-2pt] {$t_1+t_2$};
\draw [-stealth] (p2) -- ($(p2)!0.55!(s2)$) node [above] {$t_2$}; \draw (p2)--(s2);
\draw [-stealth,dashed] (s2) -- ($(s2)!0.4!(o)$) node[below left=-1pt] {$t_2$}; \draw[dashed] (s2)--(o);
\draw (pn)-- (o); \draw [-stealth] (pn) -- ($(pn)!0.3!(o)$) node [below] {$t_n$};
\draw [-stealth] (s2) -- ($(s2)!0.55!(s1)$) node [above left=-1pt] {$t_2$}; \draw (s2)--(s1);
\draw [-stealth] (o) -- ($(o)!0.75!(q)$) node [below] {$\displaystyle\sum_{i=1}^n t_i$}; \draw (o) -- (q);
\end{scope}

% Draw the points
\foreach \point in {o,p1,p2,pn,s1,s2,q} 
\path [draw=black, fill=white] (\point) circle (3pt); 

\end{tikzpicture}
\end{center}
\caption{Illustration of the proof of Theorem~\ref{thm2} for a unit ball with straight line segments on the boundary.}\label{Fig:unitball}
\end{figure}
Because of the straight line segments on the boundary of the unit ball,  $\norm{x+y}=\norm{x}+\norm{y}$ for any $x,y$ such that the unit vectors
$\frac{1}{\norm{x}}x$ and $\frac{1}{\norm{y}}y$ lie on this segment. In particular,
\begin{equation}\label{triangleineq}
\norm{s_2-s_1}+\norm{s_1-o}=\norm{s_2-o}.
\end{equation}
Now replace $p_2o$ by the edges $p_2s_2$ and $s_2 s_1$ , replace $p_1o$ by $p_1s_1$ and $s_1o$, and add the flow demand $t_2$ to $s_1o$. The change in
cost in the new Gilbert arborescence is
\begin{align*}
&\quad (w(t_1)\norm{p_1-s_1}+w(t_1+t_2)\norm{s_1-o} +w(t_2)\norm{p_2-s_2}+w(t_2)\norm{s_2-s_1}) \\
& \qquad - \left(w(t_1)\norm{p_1-o}-w(t_2)\norm{p_2-o}\right)\\
& = -w(t_1)\norm{s_1}-w(t_2)\norm{s_2}+w(t_1+t_2)\norm{s_1}+w(t_2)(\norm{s_2}-\norm{s_1}) \qquad\text{by \eqref{triangleineq}}\\
&= (w(t_1+t_2)-w(t_1)-w(t_2))\norm{s_1}\\
&= (d+h(t_1+t_2)-(d+ht_1)-(d+ht_2))\norm{s_1}\\
&= -d\norm{s_1} < 0.
\end{align*}
We have shown that a Gilbert arborescence with a Steiner point of degree at least $4$ can be decreased in cost. Hence, in an MGA a Steiner point
must necessarily be of degree $3$.
\end{proof}

\section{Degree of Steiner points in Euclidean space}%\label{section:degree}

We now consider Gilbert arborescences in Euclidean space (of arbitrary dimension) with more general weight functions, including  weight
functions of the form $w(t)=d+ht^\alpha$, where $d, h>0$ and $0<\alpha\leq 1$. Note that for all these values of $\alpha$, the weight function
$w$ satisfies all of the conditions~\eqref{eq:cost-ftn-non-neg}--\eqref{eq:cost-ftn-concave}. We show that if $0<\alpha\leq 1/2$ or $\alpha=1$,
then the maximum degree of a Steiner point is $3$, while for each $\alpha\in(1/2,1)$ we provide an example of an MGA with a Steiner point of
degree $4$. These examples are three-dimensional, and the amount of flow goes to infinity as $\alpha$ approaches $1/2$ or $1$. When $\alpha=1$
the weight function is linear, and the previous section shows that Steiner points are necessarily of degree $3$ in the Euclidean plane. We have
no examples of degree $4$ Steiner points in the Euclidean plane for higher values of $\alpha$, and we consider their existence to be highly
unlikely.

We show that there is a very general class of weight functions for which the Steiner points are necessarily of degree $3$ (Theorem~\ref{degree3}
below). Our proof is completely independent of the dimension.

We begin by reformulating Theorem~\ref{theorem:gilb-arb-charac} for Euclidean spaces in the following straightforward corollary.

\begin{corollary}\label{cor}
A Steiner point of degree $m+1$ is possible in some MGA in Euclidean space with a weight function $w(\cdot)$ that satisfies
Conditions~\eqref{eq:cost-ftn-non-neg}--\eqref{eq:cost-ftn-concave}, if and only if there exist vectors $v_1,\dots,v_m$ and flow demands $t_1,\dots,
t_m>0$ such that
\begin{gather}
\label{a} \enorm{v_i}=w(t_i), \qquad i=1,\dots,m \\
\label{b} \enorm{\sum_{i=1}^m v_i}=w(\sum_{i=1}^m t_i), \\
\label{c} \forall I\subseteq\{1,\dots,m\}\text{ with } 2\leq\card{I}\leq m-2,\quad \enorm{\sum_{i\in I} v_i}\leq w(\sum_{i\in I} t_i).
\end{gather}
\end{corollary}

For the proof of Theorem~\ref{degree3} we need to establish the following inequality valid for functions with convex derivative.

\begin{lemma}\label{ineqlemma}
If $f:[0,\infty)\to\R$ is differentiable with derivative $f'$ convex, then for all $m\geq 2$ and all $t_i\geq 0$
\textup{(}$i=1,2,\dots,m$\textup{)},
\begin{equation}\label{ineq}
 \frac{(m-1)(m-2)}{2}f(0)+\sum_{1\leq i<j\leq m} f(t_i+t_j) \leq (m-2)\sum_{i=1}^m f(t_i)+f\left(\sum_{i=1}^m t_i\right).
 \end{equation}
\end{lemma}

Note that in \eqref{ineq}, as well as in the sequel, the summation $\sum_{1\leq i<j\leq m}$ means that the sum is over all $m(m-1)/2$ pairs
$(i,j)$ that satisfy $1\leq i<j\leq m$.

\begin{proof}
We use induction on $m\geq 2$. The base case $m=2$ is trivial.

Assume now that $m\geq 3$ and that the lemma holds for $m-1$; in particular we have that
\begin{equation}\label{indhyp}
 \frac{(m-2)(m-3)}{2}f(0)+\sum_{1\leq i<j\leq m-1} f(t_i+t_j) \leq (m-3)\sum_{i=1}^{m-1} f(t_i)+f(\sum_{i=1}^{m-1} t_i).
\end{equation}
Consider $x:=t_m$ to be variable and $t_1,\dots,t_{m-1}$ fixed. Set $T:=\sum_{i=1}^{m-1}t_i$, and define
\begin{align*}
 g(x) &:= (m-2)\sum_{i=1}^m f(t_i)+f(\sum_{i=1}^m t_i)-\sum_{1\leq i<j\leq m}f(t_i+t_j)\\
      &=  (m-2)\sum_{i=1}^{m-1} f(t_i) + (m-2)f(x)+f(T+x)\\
      & \qquad\qquad-\sum_{1\leq i<j\leq m-1}f(t_i+t_j)-\sum_{i=1}^{m-1} f(t_i+x).
\end{align*}
We have to show that $g(x)\geq \frac{(m-1)(m-2)}{2}f(0)$ for all $x>0$. First of all,
\begin{align*}
g(0) &= (m-2)\sum_{i=1}^{m-1}f(t_i)+(m-2)f(0)+f(T)\\
 & \qquad\qquad -\sum_{1\leq i<j\leq m-1}f(t_i+t_j)-\sum_{i=1}^{m-1}f(t_i)\\
&=(m-3)\sum_{i=1}^{m-1}f(t_i)+f(T)-\sum_{1\leq i<j\leq m-1}f(t_i+t_j)+(m-2)f(0)\\
&\geq \frac{(m-2)(m-3)}{2}f(0)+(m-2)f(0) \quad\text{(by \eqref{indhyp})}\\
&= \frac{(m-1)(m-2)}{2}f(0).
\end{align*}
It is therefore sufficient to show that $g'(x)\geq0$ for all $x>0$. We have
\[ g'(x)=(m-2)f'(x)+f'(T+x)-\sum_{i=1}^{m-1}f'(t_i+x).\]
If $T=0$ then $t_i=0$ for all $i=1,\dots m-1$, and then $g'(x)$ is identically $0$. We may therefore assume without loss of generality that
$T>0$. Write each $x+t_j$ as a convex combination of $x$ and $x+T$:
\begin{equation*}\label{4}
t_j+x=\left(1-\frac{t_j}{T}\right) x+ \frac{t_j}{T}(x+T).
\end{equation*}
Since $f'$ is convex,
\begin{align*}
\sum_{j=1}^{m-1}f'(t_j+x) &\leq \sum_{j=1}^{m-1}\left(1-\frac{t_j}{T}\right)f'(x)+\frac{t_j}{T} f'(x+T))\\
&= (m-2)f'(x)+f'(x+T),
\end{align*}
which gives $g'(x)\geq0$. This finishes the induction step and the proof.
\end{proof}

\begin{theorem}\label{degree3}
If the weight function $w(\cdot)$ satisfies Conditions~\eqref{eq:cost-ftn-non-neg}--\eqref{eq:cost-ftn-concave}, and is differentiable with
$(w^2)'$ convex and $w(0)>0$, then all Steiner points in MGAs have degree $3$.
\end{theorem}

Note that the hypothesis is indeed satisfied for the weight function $w(t)=d+ht^\alpha$ for any $d,h>0$ and $\alpha\in[0,1/2]\cup\{1\}$, but not
when $\alpha\in(1/2,1)$.

\begin{proof}
Suppose that a Steiner point of degree $m+1\geq 3$ exists. We intend to show that $m+1=3$. Note that we do not only consider the case $m+1=4$,
since it is \emph{a priori} possible that degree $4$ Steiner points don't exist, although degree $5$ points exist (although we don't have any
examples, and it seems highly unlikely).

By Corollary~\ref{cor} there exist vectors $v_1,\dots,v_m$ and numbers $t_1,\dots,t_m>0$ that satisfy \eqref{a}, \eqref{b} and \eqref{c}. Square
\eqref{b}:
\begin{align}
\left(w(\sum_{i=1}^m t_i)\right)^2 &= \enorm{\sum_{i=1}^m v_i}^2 \notag\\
&= \sum_{i=1}^m\enorm{v_i}^2+2\sum_{1\leq i<j\leq m} \ipr{v_i}{v_j} \notag\\
&= \sum_{i=1}^m (w(t_i))^2 + 2\sum_{1\leq i<j\leq m} \ipr{v_i}{v_j}\quad\text{(by \eqref{a}).} \label{star}
\end{align}

Estimate $\ipr{v_i}{v_j}$ by applying \eqref{c} to $I=\{i,j\}$:
\begin{align*}
2\ipr{v_i}{v_j} &=\enorm{v_i+v_j}^2-\enorm{v_i}^2-\enorm{v_j}^2\\
&\leq w(t_i+t_j)^2-w(t_i)^2-w(t_j)^2 \quad\text{(again by \eqref{a}).}
\end{align*}
Sum this inequality over all pairs $(i,j)$ with $1\leq i<j\leq m$:
\begin{align*}
 2\sum_{1\leq i<j\leq m} \ipr{v_i}{v_j} &\leq \sum_{1\leq i<j\leq m} (w(t_i+t_j)^2-w(t_i)^2-w(t_j)^2) \\
 &= \sum_{1\leq i<j\leq m} (w(t_i+t_j))^2-(m-1)\sum_{i=1}^m(w(t_i))^2,
 \end{align*}
 since each $(w(t_i))^2$ is summed once for each of the $m-1$ pairs in which $i$ appears.
 Substitute this into \eqref{star}:
 \begin{equation}\label{A}
 \left(w(\sum_{i=1}^m t_i)\right)^2 \leq \sum_{1\leq i<j\leq m} (w(t_i+t_j))^2-(m-2)\sum_{i=1}^m(w(t_i))^2.
 \end{equation}
 Apply Lemma~\ref{ineqlemma} to $f=w^2$:
 \begin{equation}\label{B}
  \frac{(m-1)(m-2)}{2}(w(0))^2 + \sum_{1\leq i<j\leq m}(w(t_i+t_j))^2 \leq (m-2)\sum_{i=1}^m (w(t_i))^2+\left(w(\sum_{i=1}^m t_i)\right)^2.
 \end{equation}
 Combining \eqref{A} and \eqref{B}, we obtain $\frac{(m-1)(m-2)}{2}(w(0))^2\leq 0$, which implies $m+1\leq 3$.
\end{proof}

\bigskip
We now show that for each $\alpha\in(1/2,1)$ there exist Gilbert arborescences with Steiner points of degree $4$ if the weight function is
$w(t)=d+ht^\alpha$, with $d,h>0$ chosen appropriately.
In the example all incoming flows are equal. We first use Corollary~\ref{cor} to formulate a result for general weight functions.
\begin{proposition}\label{degree4proposition}
Let $w(\cdot)$ be a weight function that satisfies Conditions~\eqref{eq:cost-ftn-non-neg}--\eqref{eq:cost-ftn-concave}. There exists an MGA with
degree $4$ in Euclidean $3$-space with equal flow demands $t_1=t_2=t_3=:t$ and with weight function $w(\cdot)$ if, and only if
\begin{equation}\label{degree4condition}
 3w(t)^2+w(3t)^2 \leq 3w(2t)^2.
\end{equation}
\end{proposition}
\begin{proof}
By Corollary~\ref{cor}, an MGA of degree $4$ exists in Euclidean space with equal flow demands  $t_1=t_2=t_3=:t$ if, and only if, there exist three
Euclidean vectors $v_1, v_2, v_3$ such that
\begin{gather}
\norm{v_1}_2=\norm{v_2}_2=\norm{v_3}_2=w(t),\label{aa}\\
\norm{v_1+v_2}_2,\norm{v_2+v_3}_2,\norm{v_1+v_3}_2\leq w(2t),\label{bb}\\
\norm{v_1+v_2+v_3}_2=w(3t).\label{cc}
\end{gather}
(These vectors will of course span a space of dimension at most $3$.)

Square \eqref{cc} and use \eqref{aa} to obtain an expression for the sum of the three inner products:
\begin{equation}\label{iprsum}
 2(\ipr{v_1}{v_2}+\ipr{v_2}{v_3}+\ipr{v_1}{v_3})=w(3t)^2-3w(t)^2.
\end{equation}
Square \eqref{bb} and use \eqref{aa} to obtain an upper bound on each inner product $\ipr{v_i}{v_j}$ (geometrically this is a lower bound on the
angle between any two vectors):
\begin{equation*}\label{iprineq}
 2\ipr{v_i}{v_j}\leq w(2t)^2-2w(t)^2,
\end{equation*}
and substitute this into \eqref{iprsum} to obtain \eqref{degree4condition}.

Conversely, if we assume \eqref{degree4condition}, we have to find three vectors that satisfy \eqref{aa}--\eqref{cc}. Note that for each
$\lambda\in[-1/2,1]$ there exist three unit vectors $u_1,u_2,u_3\in\R^3$ such that the inner product of each pair equals $\lambda$. The one
extreme $\lambda=1$ corresponds to three equal vectors, and the other extreme $\lambda=-1/2$ to three coplanar vectors such that any two are at
an angle of $120^\circ$. If we set
\[\lambda:=\frac{w(3t)^2}{6w(t)^2}-\frac12,\] then the vectors $v_i:=w(t)u_i$ satisfy \eqref{aa}--\eqref{cc}.
It remains to show that this value of $\lambda$ really lies in the interval $[-1/2,1]$. The lower bound $\lambda\geq -1/2$ holds trivially,
while the upper bound $\lambda\leq 1$ follows from  $0\leq w(3t)\leq 3w(t)$, which in turn follows from non-negativity
(Condition~\eqref{eq:cost-ftn-non-neg}) and concavity (Condition~\eqref{eq:cost-ftn-concave}) of the cost function.
\end{proof}

\begin{corollary}
For each $\alpha\in(1/2,1)$ there exists $d,h,t>0$ and an MGA of degree $4$ in Euclidean $3$-space with cost function $w(t)=d+ht^\alpha$, and
flow demands $t_1=t_2=t_3:=t$.
\end{corollary}
\begin{proof}
By choosing the unit of the weight function appropriately, we may assume without loss of generality that $w(t)=D+t^\alpha$. We may similarly
assume that $t=1$, and then by Proposition~\ref{degree4proposition}, we only have to show that $3w(1)^2+w(3)^2\leq 3w(2)^2$ will hold for some
value of $D>0$, which is
\[ 3(D+1)^2+(D+3^\alpha)^2\leq 3(D+2^\alpha)^2,\]
or equivalently,
\[ D^2+(6+2\cdot 3^\alpha-6\cdot 2^\alpha)D+3+3^{2\alpha}-3\cdot2^{2\alpha}\leq 0.\]
A sufficient condition for this to hold for some $D>0$, is that the quadratic polynomial in $D$ on the left has a positive root. For this to
hold it is in turn sufficient that its constant coefficient is negative, i.e., that
\[f(\alpha)=3+3^{2\alpha}-3\cdot2^{2\alpha}< 0 \text{ for all }\alpha\in(1/2,1).\]
However, it is easily checked that $f(1/2)=f(1)=0$ and that $f''(\alpha)>0$ for all $\alpha\in(1/2,1)$, so that $f$ is convex on $(1/2,1)$. It
follows that $f$ is negative on $(1/2,1)$, which finishes the proof.
\end{proof}

\section{Conclusion}
In this paper we have studied the problem of designing a minimum cost flow network interconnecting $n$ sources and a single sink, each with
known locations and flow demands, in general finite-dimensional normed spaces. The network may contain other unprescribed nodes, known as Steiner
points. For concave increasing cost functions, a minimum cost network of this sort has a tree topology, and hence can be called a Minimum
Gilbert Arborescence (MGA). We have characterised the local topological structure of Steiner points in MGAs for linear weight functions,
specifically showing that Steiner points necessarily have degree $3$, and we have studied the degree of Steiner points in Euclidean spaces (of
arbitrary dimension) for a more general class of weight functions.

\section*{Acknowledgments} This research was supported by an ARC Linkage Grant with Newmont Australia Limited and The University of Melbourne.
Part of this research was done while Marcus Volz was a PhD student at The University of Melbourne. Much of this paper was written while Konrad
Swanepoel was visiting the Department of Mechanical Engineering of The University of Melbourne on a Tewkesbury Fellowship.

\end{document}